\documentclass[11pt]{article}

\usepackage[T1]{fontenc}
\usepackage{lmodern}
\usepackage[margin=1.12in]{geometry}
\usepackage{amsmath,amssymb,amsthm,mathtools}
\usepackage{microtype}
\usepackage[hidelinks]{hyperref}
\hypersetup{
  pdftitle={Small Counterexamples to the Gaussian Moments Conjecture},
  pdfauthor={Christopher D. Long},
  pdfsubject={Explicit counterexamples to the Gaussian Moments Conjecture},
  pdfkeywords={Gaussian Moments Conjecture, Jacobian Conjecture, Gaussian moments, Lagrange inversion}
}
\newcommand{\C}{\mathbb C}
\newcommand{\E}{\mathbb E}
\newcommand{\GMC}{\mathrm{GMC}}
\newcommand{\SIC}{\mathrm{SIC}}

\newtheorem{theorem}{Theorem}[section]
\newtheorem{proposition}[theorem]{Proposition}
\newtheorem{corollary}[theorem]{Corollary}
\theoremstyle{remark}
\newtheorem{remark}[theorem]{Remark}

\title{Small Counterexamples to the Gaussian Moments Conjecture}
\author{Christopher D. Long\\\texttt{galizur@gmail.com}}
\date{July 20, 2026}

\begin{document}
\maketitle

\begin{abstract}
We give explicit complex polynomials $P,Q$ in three independent standard real Gaussian variables such that
\[
  \E(P^m)=0,\qquad \E(QP^m)=m!\neq0
\]
for every $m\geq1$.  In natural complex linear coordinates, $P$ has five terms and total degree $4$.  Hence the Gaussian Moments Conjecture is false in every dimension $n\geq3$.  We also give a six-term cubic example in four variables, which was found first and already proves failure for every $n\geq4$.  Both examples follow from the same coefficient identity.  The search was prompted by Levent Alp\"oge's public announcement of an explicit three-dimensional counterexample to the Jacobian Conjecture.  Although the main theorem of Derksen, van den Essen, and Zhao is stated globally in dimension, its proof has fixed-dimensional content: a noninvertible cubic-homogeneous Keller map in $r$ variables forces the failure of $\GMC(2r)$.  Tracking a standard Bass--Connell--Wright reduction of the announced map gives a conservative cubic-homogeneous counterexample in $79$ variables, and hence a route-based failure of $\GMC(158)$.  That route is nonconstructive at the final Gaussian step and does not furnish explicit polynomials $P,Q$.  The much smaller explicit failures in dimensions $4$ and $3$ below were not derived from the announced Jacobian map.
\end{abstract}



\section{Introduction}

Let $X=(X_1,\dots,X_n)$ be a vector of independent standard real Gaussian random variables.  Derksen, van den Essen, and Zhao formulated the Gaussian Moments Conjecture $\GMC(n)$ as follows: if $P\in\C[x_1,\dots,x_n]$ satisfies
\[
  \E\bigl(P(X)^m\bigr)=0\qquad(m\geq1),
\]
then, for every $Q\in\C[x_1,\dots,x_n]$,
\[
  \E\bigl(Q(X)P(X)^m\bigr)=0
\]
for all sufficiently large $m$ \cite[Conjecture~1.1]{DVEZ}.  Complex coefficients are essential to the formulation: for real coefficients, the condition at $m=2$ already forces $P=0$.

The logical relation with the Jacobian Conjecture requires some care.  Derksen, van den Essen, and Zhao proved \cite[Theorem~1.6]{DVEZ} that
\[
  \bigl(\forall n\geq1,\ \GMC(n)\bigr)
  \quad\Longrightarrow\quad
  \mathrm{JC},
\]
where $\mathrm{JC}$ denotes the Jacobian Conjecture in all dimensions.  The latter conjecture was proposed by Keller \cite{Keller}.  The contrapositive of this displayed theorem, taken in isolation, gives only
\begin{equation}\label{eq:global-logic}
  \neg\mathrm{JC}
  \quad\Longrightarrow\quad
  \exists n\geq2\text{ such that }\neg\GMC(n),
\end{equation}
where the restriction $n\geq2$ uses the known truth of $\GMC(1)$ \cite[Proposition~4.2]{DVEZ}.

\medskip
\noindent\textbf{Fixed-dimensional content of the proof.}
The proof of the theorem is stronger than its global formulation.  Let $\SIC(r)$ denote Zhao's Special Image Conjecture in $r$ variables.  Propositions~3.2 and~3.1 of \cite{DVEZ} give, for each fixed $r$,
\[
  \GMC(2r)\quad\Longrightarrow\quad\SIC(r).
\]
For a fixed cubic-homogeneous Keller map $F=z-H$ in $r$ variables, the proof of \cite[Theorem~3.7]{Zhao} shows that $\SIC(r)$ implies that $F$ is invertible.  Consequently,
\begin{equation}\label{eq:fixed-logic}
  \left.
  \begin{array}{c}
  F=z-H:\C^r\longrightarrow\C^r,\\
  H\text{ cubic homogeneous},\quad\det JF=1,\\
  F\text{ not invertible}
  \end{array}
  \right\}
  \quad\Longrightarrow\quad
  \neg\GMC(2r).
\end{equation}
Here $JH$ denotes the Jacobian matrix.  This implication identifies a Gaussian dimension after a cubic-homogeneous reduction, but it need not produce explicit Gaussian witness polynomials: Proposition~3.1 of \cite{DVEZ} uses a countable-union argument over Zariski-closed sets.

On July 20, 2026, Levent Alp\"oge publicly announced an explicit three-dimensional counterexample \cite{Alpoge}.  Kevin Buzzard's contemporaneous account identifies Akhil as Akhil Mathew, reports that Mathew suggested the problem to Alp\"oge, and states that the announcement credited Fable with finding the counterexample \cite{Buzzard}.  Buzzard also records that Paul Lezeau manually formalized the example in Lean; Lezeau submitted the formalization to Google DeepMind's \emph{Formal Conjectures} repository \cite{Lezeau}.  The input dimension $3$ does not itself transfer to the index in $\GMC(n)$, but the standard homogeneous reduction can be tracked.  Section~\ref{sec:jc-bound} gives the conservative bound $\neg\GMC(158)$ from this route.  The explicit examples presented later are substantially smaller and were not obtained from the announced map.

The first construction gives a cubic counterexample in four real Gaussian variables, proving $\GMC(n)$ false for every $n\geq4$.  The second replaces one complex Gaussian pair by one real Gaussian variable and yields a quartic counterexample in three variables.  Thus the final conclusion is
\[
  \boxed{\GMC(n)\text{ is false for every }n\geq3.}
\]

\section{A tracked bound from the announced Jacobian map}\label{sec:jc-bound}

We record a quantitative consequence because it distinguishes the global statement of \cite[Theorem~1.6]{DVEZ} from the fixed-dimensional content of its proof.  In the normalization used in Lezeau's formalization \cite{Lezeau}, the announced map is
\[
\begin{aligned}
 F_1&=(1+2xy)^3z+4y^2(1+2xy)(2+3xy),\\
 F_2&=y+3x(1+2xy)^2z+12xy^2(2+3xy),\\
 F_3&=-x+3x^2y+x^3z.
\end{aligned}
\]
It has Jacobian determinant $1$ and identifies the distinct points
\[
  \left(1,-\frac34,\frac{13}{4}\right),\qquad
  \left(-1,\frac34,\frac{13}{4}\right).
\]
Composing on the target with the linear automorphism $(u,v,w)\mapsto(-w,v,u)$ gives a normalized map $G$ with linear part the identity:
\[
\begin{aligned}
G_1={}&x-3x^2y-x^3z,\\
G_2={}&y+3xz+24xy^2+12x^2yz+36x^2y^3+12x^3y^2z,\\
G_3={}&z+8y^2+6xyz+28xy^3+12x^2y^2z+24x^2y^4+8x^3y^3z.
\end{aligned}
\]
Among its terms of degree at least $4$, there are respectively $3,2,2,1$ monomials of degrees $4,5,6,7$.

\begin{proposition}\label{prop:route-bound}
The announced three-variable Jacobian counterexample, together with the standard degree and homogeneous reductions, implies $\neg\GMC(158)$.
\end{proposition}

\begin{proof}
We track the elementary reduction in \cite[Section~3]{BCW}.  Let $c(d)$ be an upper bound for the number of degree-lowering steps required for one monomial of degree $d$, with $c(d)=0$ for $d\leq3$.  If a degree-$d$ monomial is factored as $PQ$, where $\deg P=p$, $\deg Q=q$, $p+q=d$, and $p,q\leq d-2$, one Bass--Connell--Wright step adds two variables and replaces it by terms of degrees at most
\[
  p+1,\quad q+1,\quad p,\quad q,\quad 2.
\]
Thus
\[
  c(d)\leq1+c(p+1)+c(q+1)+c(p)+c(q).
\]
Using the balanced choices $4=2+2$, $5=2+3$, $6=3+3$, and $7=3+4$ gives
\[
  c(4)=1,\qquad c(5)\leq2,\qquad c(6)\leq3,
  \qquad c(7)\leq5.
\]
The displayed support of $G$ therefore requires at most
\[
  3c(4)+2c(5)+2c(6)+c(7)\leq18
\]
steps.  Since each step adds two variables, this produces a noninvertible Keller map of degree at most $3$ with identity linear part in
\[
  s=3+2\cdot18=39
\]
variables.

For completeness, the next dimension count can be read directly from the construction in \cite[Section~4]{BCW}; its optional multilinearization is not needed here.  Write the resulting map as
\[
  K(X)=X+K^{(2)}(X)+K^{(3)}(X),\qquad X\in\C^s.
\]
In variables $(X,Y)\in\C^{2s}$, the map
\[
  U(X,Y)=\bigl(X+K^{(2)}(X)+Y,\;Y-K^{(3)}(X)\bigr)
\]
is stably equivalent to $K$, has the form $(X,Y)+N(X,Y)$ with $\deg N\leq3$, and has nilpotent $JN$.  One direct verification uses
\[
 E_t(X)=X+tK^{(2)}(X)+t^2K^{(3)}(X)
\]
and observes that $(X,Y)+tN$ is obtained from the stabilization of $E_t$ by elementary pre- and post-compositions.  Hence $I+tJN$ is invertible over the polynomial ring in $t$, which forces $JN$ to be nilpotent.  Finally, the Bass--Connell--Wright homogenization adds one variable and replaces the degree-$1$, $2$, and $3$ parts of $N$ by
\[
  N^{(1)}T^2+N^{(2)}T+N^{(3)}.
\]
This gives a noninvertible map equal to the identity plus a cubic-homogeneous map in
\[
  r=2s+1=79
\]
variables.  Applying \eqref{eq:fixed-logic} gives $\neg\GMC(2r)=\neg\GMC(158)$.
\end{proof}

\begin{remark}
The number $158$ is a transparent upper bound, not an optimization and not a claim about the least failing Gaussian dimension.  The reduction depends on the degree and monomial support of the particular Jacobian map; it does not give a bound depending only on the original input dimension $3$.  It also does not supply explicit Gaussian witness polynomials $P,Q$.
\end{remark}

\section{Gaussian contractions}

Let $X,Y$ be independent standard real Gaussians and put
\[
  Z=\frac{X+iY}{\sqrt2},\qquad W=\frac{X-iY}{\sqrt2}.
\]
Then $W=\overline Z$ as random variables, while both are complex linear polynomials in $X,Y$.  Rotational invariance gives
\begin{equation}\label{eq:contraction}
  \E(W^aZ^b)=\delta_{ab}\,a!\qquad(a,b\geq0).
\end{equation}
Equivalently, for every polynomial $R\in\C[z]$,
\begin{equation}\label{eq:coefficient-contraction}
  \E\bigl(W^aR(Z)\bigr)=a!\,[z^a]R(z).
\end{equation}
If $T$ is an independent standard real Gaussian, then
\begin{equation}\label{eq:real-gaussian}
  \E(T^{2k})=(2k-1)!!=\frac{(2k)!}{2^k k!},
\end{equation}
and hence, as a formal ordinary generating function,
\begin{equation}\label{eq:central-binomial}
  \sum_{k\geq0}\frac{(-1)^k}{4^k}\binom{2k}{k}u^k=(1+u)^{-1/2}.
\end{equation}
All exponential generating functions below are formal identities; no analytic integrability of $e^{tP}$ is asserted or needed.

\section{A cubic counterexample in four variables}

Let $X_1,Y_1,X_2,Y_2$ be independent standard real Gaussians and define
\[
  Z_j=\frac{X_j+iY_j}{\sqrt2},\qquad
  W_j=\frac{X_j-iY_j}{\sqrt2}\qquad(j=1,2).
\]
Set
\begin{equation}\label{eq:P4}
  P_4=(1+Z_2)\bigl(W_1(1-Z_1)+W_2\bigr),
  \qquad Q_4=Z_2.
\end{equation}
In these coordinates $P_4$ has six terms and total degree $3$.

\begin{proposition}\label{prop:four}
For every $A\in\C[z]$ and $m\geq1$,
\begin{equation}\label{eq:master-four}
  \E\bigl(A(Z_2)P_4^m\bigr)
  =m!\,[z^m]A(z)(1+z)^{m-1}.
\end{equation}
In particular,
\[
  \E(P_4^m)=0,
  \qquad
  \E(Q_4P_4^m)=m!
  \qquad(m\geq1).
\]
\end{proposition}

\begin{proof}
Expanding according to the number $a$ of factors chosen from $W_1(1-Z_1)$ gives
\begin{align*}
\E\bigl(A(Z_2)P_4^m\bigr)
 &=\sum_{a=0}^m\binom ma
   \E\bigl(W_1^a(1-Z_1)^a\bigr)
   \E\bigl(W_2^{m-a}A(Z_2)(1+Z_2)^m\bigr)\\
 &=m!\sum_{a=0}^m(-1)^a
   [z^{m-a}]A(z)(1+z)^m,
\end{align*}
where \eqref{eq:coefficient-contraction} was used twice.  Therefore
\begin{align*}
\frac1{m!}\E\bigl(A(Z_2)P_4^m\bigr)
 &= [z^m]A(z)(1+z)^m\sum_{a\geq0}(-z)^a\\
 &= [z^m]A(z)(1+z)^{m-1}.
\end{align*}
Terms with $a>m$ cannot affect the coefficient of $z^m$.  Taking $A=1$ and $A=z$ gives the stated moments.
\end{proof}

Equivalently,
\begin{equation}\label{eq:egf-four}
  \E(e^{tP_4})=1,
  \qquad
  \E(Q_4e^{tP_4})=\frac{t}{1-t}.
\end{equation}

\section{A quartic counterexample in three variables}

Let $X_1,X_2,T$ be independent standard real Gaussians and put
\[
  Z=\frac{X_1+iX_2}{\sqrt2},\qquad
  W=\frac{X_1-iX_2}{\sqrt2}.
\]
Define
\begin{equation}\label{eq:P3}
\begin{split}
  P_3
  &=(1+Z)\left(W-\frac12(2+Z)T^2\right)\\
  &=W+WZ-T^2-\frac32ZT^2-\frac12Z^2T^2,
  \qquad Q_3=Z.
\end{split}
\end{equation}
Thus $P_3$ has five terms and total degree $4$.

\begin{theorem}\label{thm:three}
For every $A\in\C[z]$ and $m\geq1$,
\begin{equation}\label{eq:master-three}
  \E\bigl(A(Z)P_3^m\bigr)
  =m!\,[z^m]A(z)(1+z)^{m-1}.
\end{equation}
Consequently,
\begin{equation}\label{eq:counterexample-three}
  \E(P_3^m)=0,
  \qquad
  \E(Q_3P_3^m)=m!\neq0
  \qquad(m\geq1).
\end{equation}
\end{theorem}

\begin{proof}
Write
\[
  h(z)=1+z,
  \qquad
  v(z)=-\frac12(1+z)(2+z),
\]
so that $P_3=Wh(Z)+v(Z)T^2$.  Expanding according to the number $k$ of factors chosen from $v(Z)T^2$, and using \eqref{eq:coefficient-contraction} and \eqref{eq:real-gaussian}, gives
\begin{align*}
\frac1{m!}\E\bigl(A(Z)P_3^m\bigr)
 &=\sum_{k=0}^m
   \frac{(-1)^k}{4^k}\binom{2k}{k}
   [z^{m-k}]A(z)(1+z)^m(2+z)^k\\
 &= [z^m]A(z)(1+z)^m
    \sum_{k\geq0}\frac{(-1)^k}{4^k}\binom{2k}{k}
      \bigl(z(2+z)\bigr)^k.
\end{align*}
Extending the sum beyond $k=m$ does not change the coefficient of $z^m$.  By \eqref{eq:central-binomial} and
\[
  1+z(2+z)=(1+z)^2,
\]
the last expression is
\[
  [z^m]A(z)(1+z)^m\bigl((1+z)^2\bigr)^{-1/2}
  =[z^m]A(z)(1+z)^{m-1},
\]
where the square root is the formal branch with constant term $1$.  Taking $A=1$ and $A=z$ proves \eqref{eq:counterexample-three}.
\end{proof}

Thus
\begin{equation}\label{eq:egf-three}
  \E(e^{tP_3})=1,
  \qquad
  \E(Q_3e^{tP_3})=\frac{t}{1-t}.
\end{equation}
The failure of eventual vanishing is maximal: the mixed moment is nonzero for every positive exponent.

\begin{corollary}\label{cor:dimensions}
The Gaussian Moments Conjecture $\GMC(n)$ is false for every $n\geq3$.
\end{corollary}

\begin{proof}
Theorem~\ref{thm:three} disproves $\GMC(3)$.  For $n>3$, use the same polynomials in the first three variables and ignore the remaining independent Gaussian variables.
\end{proof}

Derksen, van den Essen, and Zhao proved $\GMC(1)$ \cite[Proposition~4.2]{DVEZ}.  The only remaining dimension is therefore $n=2$.

\section{How the examples were found}

The four-variable example arose from the complex-Gaussian realization in \cite[Proposition~3.2]{DVEZ}.  Let $Z_1,\dots,Z_d$ be independent standard circular complex Gaussians and put $W_j=\overline{Z_j}$.  For a polynomial map $H=(H_1,\dots,H_d)$, consider
\[
  P=W\mathbin{\cdot}H(Z)=\sum_{j=1}^d W_jH_j(Z).
\]
A form of Lagrange--Good inversion \cite{Good} gives
\begin{equation}\label{eq:good}
  \E\bigl(A(Z)e^{tP}\bigr)
  =\frac{A(\boldsymbol g(t))}{\det\bigl(I-tJH(\boldsymbol g(t))\bigr)},
  \qquad \boldsymbol g(t)=tH(\boldsymbol g(t)),
\end{equation}
where $JH$ is the Jacobian matrix and $\boldsymbol g(t)\in t\C[[t]]^d$ is the unique formal solution.
The search therefore became: make $\boldsymbol g(t)$ nonpolynomial, but force the determinant in \eqref{eq:good} to equal $1$ along this inverse branch.  The choice
\[
  H_1(z_1,z_2)=(1-z_1)(1+z_2),
  \qquad H_2(z_1,z_2)=1+z_2
\]
gives
\[
  \boldsymbol g(t)=\left(t,\frac{t}{1-t}\right),
  \qquad
  \det\bigl(I-tJH(g(t))\bigr)=1.
\]
This is precisely $P_4$, and $A(z_1,z_2)=z_2$ reads off $t/(1-t)$.

The three-variable example replaces one circular complex Gaussian pair by one real Gaussian direction.  For
\[
  P=Wh(Z)+v(Z)T^2,
\]
univariate Lagrange inversion and the identity $\E(e^{sT^2})=(1-2s)^{-1/2}$ give
\begin{equation}\label{eq:half-pair}
  \E\bigl(A(Z)e^{tP}\bigr)
  =\frac{A(\zeta(t))\bigl(1-2t v(\zeta(t))\bigr)^{-1/2}}
         {1-th'(\zeta(t))},
  \qquad \zeta(t)=t h(\zeta(t)).
\end{equation}
Taking $h(z)=1+z$ gives $\zeta=t/(1-t)$ and denominator $1-t$.  The choice
\[
  v(z)=-\frac12(1+z)(2+z)
\]
forces
\[
  1-2t v(\zeta(t))=(1-t)^{-2},
\]
so the square-root factor is $1-t$ and cancels the denominator.  This yields $P_3$.  Informally, a single real Gaussian supplies half of the determinant correction supplied by a circular complex pair.

The announced Jacobian counterexample supplied the initial motivation.  As Proposition~\ref{prop:route-bound} records, tracking the known reduction gives the fixed-dimensional but nonexplicit conclusion $\neg\GMC(158)$.  The much stronger conclusions $\neg\GMC(4)$ and $\neg\GMC(3)$ come instead from the explicit polynomials above.  No coordinate, term, or algebraic feature of the announced map was used in either construction.

\section{A brief remark on dimension two}

Let $Z=(X+iY)/\sqrt2$, $W=(X-iY)/\sqrt2$, and $U=ZW$.  Give $Z$ weight $1$ and $W$ weight $-1$.  Gaussian expectation kills all nonzero weights, while on $\C[U]$ it is the linear factorial functional $\mathcal L:\C[U]\to\C$ defined by $\mathcal L(U^j)=j!$.

This immediately excludes several natural sources of a two-dimensional counterexample.  Polynomials supported only on positive weights or only on negative weights satisfy the desired eventual vanishing by weight.  Homogeneous polynomials are covered by \cite[Corollary~4.4]{DVEZ}, using the theorem of Duistermaat and van der Kallen \cite{DvK}.  On $\C[U]$, the hypothesis $\mathcal L(P^m)=0$ for all $m$ forces $P=0$ by the one-variable Factorial Conjecture \cite[Theorem~4.9]{VEWZ}.  More generally, if
\[
  P=ZA(U)+WB(U),
\]
then odd moments vanish by weight and
\[
  \E(P^{2r})=\binom{2r}{r}\mathcal L\bigl((UA(U)B(U))^r\bigr).
\]
The same theorem forces $A=0$ or $B=0$, reducing again to one-sided weights.  Thus any counterexample to $\GMC(2)$ would require a richer inhomogeneous mixture of positive and negative weights.  We do not settle this case.

\section*{AI provenance and author responsibility}

The four-variable construction was produced by ChatGPT 5.6 Sol Pro, without human intervention after the initial prompt, after it was told that the Jacobian Conjecture had been disproved and asked whether a small counterexample to the Gaussian Moments Conjecture might therefore exist.  After being shown the four-variable example, Claude Fable 5 found the three-variable construction and supplied independent algebraic checks and structural comments.  ChatGPT also assisted with editing this draft.  The human author bears full responsibility for the mathematics, the exposition, and every claim in this manuscript.

\end{document}